\documentclass[10pt,a4paper]{article}
\usepackage{amsthm}
\usepackage[fleqn]{amsmath}
\usepackage{amssymb}
\usepackage{amsfonts}
\newtheorem{theorem}{Theorem}
\newtheorem{lemma}{Lemma}
\newtheorem{Cor}{Corollary}
\newtheorem{obs}{Observation}

\begin{document}
\title{The size of maximal systems of brick islands}
\author{Tom Eccles\thanks{Department of Pure Mathematics and Mathematical Statistics, Centre for Mathematical
Sciences, Wilberforce Road, Cambridge CB3 0WB, England.}}
\maketitle
\begin{abstract}
For integers $m_1,...,m_d>0$ and a cuboid $M=[0,m_1]\times ... \times [0,m_d]\subset \mathbb{R}^d$, a brick of $M$ is a closed cuboid whose vertices have integer coordinates. A set $H$ of bricks in $M$ is a system of brick islands if for each pair of bricks in $H$ one contains the other or they are disjoint. Such a system is maximal if it cannot be extended to a larger system of brick islands. Extending the work of Lengv\'{a}rszky, we show that the minimum size of a maximal system of brick islands in $M$ is $\sum_{i=1}^d m_i - (d-1)$. Also, in a cube $C=[0,m]^d$ we define the corresponding notion of a system of cubic islands, and prove bounds on the sizes of maximal systems of cubic islands.
\end{abstract}
\section{Introduction}

\indent The concept of a \emph{system of rectangular islands} was introduced by Cz\'{e}dli in \cite{Cze}. In \cite{Plu}, Pluh\'{a}r generalised this concept to that of a \emph{system of brick islands} in higher dimensions, a direction mentioned both in \cite{Cze} and by Lengv\'{a}rszky in \cite{Len1}. To introduce these concepts, let $M=[0,m_1]\times ... \times [0,m_d]\subset \mathbb{R}^d$  be a closed cuboid. Then a \emph{brick} of $M$ is a set of the form $[a_1,b_1]\times...[a_d,b_d]$, where for each $1\leq i\leq d$ we have $0\leq a_i < b_i \leq m_i$, and $a_i$, $b_i\in \mathbb{Z}$. A system of brick islands in $M$ is a set $H$ of bricks in $M$ such that whenever $M_1$, $M_2\in H$, either $M_1\subseteq M_2$, $M_2\subseteq M_1$ or $M_1 \cap M_2 = \emptyset$. We denote the set of systems of brick islands in $M$ by $I_M$, and the maximal elements of $I_M$ with respect to inclusion by $Max(I_M)$. When $M$ is $2$-dimensional, a system of brick islands can also be called a system of rectangular islands.

A related concept is that of a \emph{system of square islands}, introduced by \linebreak Lengv\'{a}rszky in \cite{Len2}. For $m$ a positive integer, let $S=[0,m]\times [0,m]$ be a closed square in the plane. A system of square islands in $S$ is a system of rectangular islands $H$ with every rectangle in $H$ being a square. We denote the set of these systems by $I_S$, and the maximal elements of $I_S$ with respect to inclusion by $Max(I_S)$.

We define the higher dimensional analogue of systems of square islands, as suggested in \cite{Len2}. Let $C=[0,m]^d$ be a closed cube in $d$-dimensional space. Then we define a \emph{system of cubic islands} in $C$ to be a system $H$ of brick islands in $C$ such that each brick in $H$ is a cube. We denote the set of these systems by $I_C$, and the maximal elements of $I_C$ with respect to inclusion by $Max(I_C)$.

We shall be concerned with the possible cardinalities of maximal systems of brick and cubic islands. For a cuboid $M=[0,m_1]\times ... \times [0,m_d]$, we define
\begin{equation*}
 f_d(m_1,...,m_d)=max\{\left|H\right|:H\in Max(I_M)\}
\end{equation*}
and 
\begin{equation*}
 g_d(m_1,...,m_d)=min\{\left|H\right|:H\in Max(I_M)\}.
\end{equation*}
Similarly, for a cube $C=[0,m]^d$, we define
\begin{equation*}
 f'_d(m)=max\{\left|H\right|:H\in Max(I_C)\}
\end{equation*}
and
\begin{equation*}
 g'_d(m)=min\{\left|H\right|:H\in Max(I_C)\}.
\end{equation*}

\section{Earlier work}
We summarise the main results of Cz\'{e}dli \cite{Cze}, Lengv\'{a}rszky \cite{Len1}, \cite{Len2} and Pluh\'{a}r \cite{Plu} that relate to our work. All these results concern the possible cardinalities of maximal systems of rectangular, square or brick islands.
In \cite{Cze}, Cz\'{e}dli proved that for systems of rectangular islands in $M=[0,m_1]\times[0,m_2]$,
\begin{equation*}
f_2(m_1,m_2)=\left\lfloor \frac{m_1m_2+m_1+m_2-1}{2}\right\rfloor.
\end{equation*}
In \cite{Len1} and \cite{Len2}, Lengv\'{a}rszky proved
\begin{equation*}
g_2(m_1,m_2)=m_1+m_2-1,
\end{equation*}
and for systems of square islands in $S=[0,m]\times[0,m]$,
\begin{align*}
g_2'(m)&=m \text{, and}\\
f_2'(m)&\leq \frac{m(m+2)}{3}
\end{align*}
with equality in the last being achieved for $k$ a positive integer and $m=2^k-1$.
In \cite{Plu}, Pluh\'{a}r proved that for systems of brick islands in $M=[0,m_1]\times ... \times [0,m_d]$,
\begin{equation*}
\frac{m_1m_2...m_d+\sum m_{j_1}...m_{j_{d-1}}}{2^{d-1}}-1 \leq f_d(m_1,...,m_d)\leq \frac{(m_1+1)...(m_d+1)}{2^{d-1}}-1
\end{equation*}
where the sum in the lower bound runs over the $d-1$ element subsets of \{1,...,d\}.
\section{Our results}
Using methods similar to those employed in \cite{Len1} and \cite{Len2}, we shall prove the following theorems about the possible cardinalities of maximal systems of brick and cubic islands.
\begin{theorem}\label{mincuboids}
Let $M=[0,m_1]\times...\times [0,m_d]\subset \mathbb{R}^d$ be a cuboid. Then the minimal size of a maximal system of cuboid islands in $M$, is given by
\begin{equation*}
g_d(m_1,...,m_d)=\sum_{i=1}^d m_i -(d-1).
\end{equation*}
\end{theorem}

\begin{theorem}
\label{mincubes}
Let $C=[0,m]^d$ be a cube in $d$-dimensional space. Then the minimal size of a maximal system of cubic islands in $C$ is given by
\begin{equation*}
g'_d(m)=m.
\end{equation*}
\end{theorem}

\begin{theorem}
\label{maxcubes}
Let $C=[0,m]^d$ be a cube in $d$-dimensional space. Then the maximal size of a system of cubic islands in $C$, is bounded by
\begin{equation*}
f'_d(m)\leq \frac{(m+1)^d-1}{2^d-1}.
\end{equation*}
Moreover, equality can be achieved when $m=2^k-1$ for some positive integer $k$.
\end{theorem}
\noindent
The rest of this paper will be organised as follows. In section \ref{upperbound}, we shall prove the upper bound for Theorem \ref{mincuboids}. In section \ref{prelim}, we shall make some preliminary observations which will help us in the proof of the lower bound. Then we shall prove the lower bound, in section \ref{maincuboidproof}. In section \ref{extremal}, we shall classify the minimal members of $Max(I_M)$ for a cuboid $M$. In sections \ref{mincubeproof} and \ref{maxcubeproof}, we shall prove Theorems \ref{mincubes} and \ref{maxcubes} respectively.

\section{The upper bound in Theorem \ref{mincuboids}}
\label{upperbound}
To establish one direction of Theorem \ref{mincuboids}, we show that
\begin{equation*}
g_d(m_1,...,m_d)\leq \sum_{i=1}^d m_i -(d-1)
\end{equation*}
by exhibiting a system of brick islands of this size. Indeed, for $M=[0,m_1]\times...\times [0,m_d]$, define a set of bricks $H$ by
\begin{equation*}
H = \{[0,m_1]\times ... \times [0,m_{i-1}]\times [0,n_i] \times [0,1] \times ... \times [0,1]: 1\leq i \leq d, \text{ } 1\leq n_i \leq m_i\}.
\end{equation*}
This defines a system of $\sum_{i=1}^d m_i -(d-1)$ nested bricks. Since each of these bricks extends the last by $1$ in one dimension, $H$ is a maximal system of brick islands in $M$, which establishes our upper bound on $g_d$.

\section{Preliminary results}
\label{prelim}
Working towards the lower bound for $g_d$, we shall start with some observations about maximal systems of brick islands. For $M=[0,m_1]\times...\times [0,m_2]$, we define an \emph{elementary cube} in $M$ to be a cube of the form $[a_1,a_1+1]\times...\times[a_d,a_d+1]$, where for each $1\leq i \leq d$, $a_i \in \mathbb{Z}$ and $0\leq a_i \leq m_i-1$. We call this cube the elementary cube based at $(a_1,a_2,...,a_d)$.

\begin{obs}
Suppose $I\subset[n]$ and $A$ is an elementary cube based at $(a_1,...,a_d)$ such that $a_i = 0$ whenever $i\in I$. Suppose $A'$ is another elementary cube, based at $(a'_1,...,a'_d)$. If $a'_i = a_i$ for $i\notin I$, and $a'_i \in \{0,1\}$ for $i\in I$, then every elementary cube that intersects $A$ also intersects $A'$.
\end{obs}
For a system of brick islands $H$ in $M$, let $Max(H)$ be the set of maximal elements of $H\backslash\{M\}$ with respect to inclusion.
\begin{Cor}
\label{nhds}
Let $M=[0,m_1]\times...\times [0,m_d]$ and let $H\in Max(I_M)$.  Suppose that $\left|Max(H)\right|>1$ and $R=[r_{1,1},r_{1,2}]\times ... \times [r_{d,1},r_{d,2}]\in Max(H)$. Then no $r_{i,1}$ is $1$, and no $r_{i,2}$ is $m_i-1$,
\end{Cor}
\begin{proof}
Indeed, suppose $r_{1,1}=1$. Let $R'$ be $[0,r_{1,2}]\times ... \times [r_{d,1},r_{d,2}]$. From the observation above, any elementary cube intersecting $R'$ intersects $R$, and hence no element of $Max(H)\backslash\{R\}$ intersects $R'$. The brick $R'$ cannot be in $H$ already as then we would have $R'=M$ and $\left|Max(H)\right| = 1$. This shows that $H$ is not maximal, since we can add $R'$ to it, which is a contradiction. 
\end{proof}
\begin{Cor}
Let $M=[0,m_1]\times...\times [0,m_d]$ and let $H\in Max(I_M)$. If $\left| Max(H)\right| > 1$, every vertex of $M$ is occupied by a member of $Max(H)$.
\end{Cor}
\begin{proof}
Given a vertex $v$ of the cuboid $M$, let $C_v$ be the elementary cube which contains $v$. As $H$ is maximal, $Max(H)$ contains some brick $R$ which intersects $C_v$. By the previous result, $R$ must contain $C_v$.
\end{proof}

\begin{Cor}
\label{gaps}
Let $M=[0,m_1]\times...\times [0,m_d]$ and let $H\in Max(I_M)$. Suppose $\left| Max(H)\right| > 1$ and $R_1$, $R_2$ are bricks in $Max(H)$ which intersect an edge $E$ of the cube. There is some section of $E$ between the intersections of $R_1$ and $R_2$ with $E$ - we shall call this the gap between $R_1$ and $R_2$ on $E$. Suppose that no other member of $Max(I_M)$ intersects this gap. Then the length of the gap is at most $2$. Further, if the length of the gap is exactly $2$, neither of $R_1$, $R_2$ is an elementary cube.
\end{Cor}
\begin{proof}
We may assume that $E=\{(x,0,...,0):0\leq x \leq m_1\}$. Suppose there is a gap of at least $3$ between $R_{1}$, $R_{2}$ on $E$ - so no member of $Max(H)$ intersects $\{(x,0,0,...,0):a<x<a+3\}$, for some integer $1 \leq a \leq n-4$. Then, by applying Corollary \ref{nhds} three times, we see that the elementary cube based at $(a+1,0,0,...0)$ intersects no member of $Max(H)$ - otherwise this member of $Max(H)$ would also intersect the gap between $R_1$ and $R_2$ on $E$. This gives rise to a contradiction - $H$ is not maximal, as we can add this elementary cube to it.

Now, suppose we have a gap of length $2$ between $R_1$ and $R_2$ on $E$ - so that $(a,0,...0)\in R_1$, $(a+2,0,...,0)\in R_2$, and no member of $Max(H)$ intersects $\{(x,0,0,...,0):a<x<a+2\}$, for some integer $1 \leq a \leq n-3$. If $R_1$ was an elementary cube, then we extend it to $R_1'$ by adding in the elementary cube based at $(a,0,0,...0)$. By Corollary \ref{nhds} applied to this elementary cube and the elementary cube based at $(a+1,0,...0)$, $R_1'$ intersects no elements of $Max(H)$ other than $R$. This shows that we can add $R_1'$ into $H$, contradicting its maximality.
\end{proof}
\begin{obs}
\label{subsystems}

Let $M=[0,m_1]\times...\times [0,m_d]$, $H\in Max(I_M)$ and $R\in H$. Then those bricks in $H$ which are contained in $R$ form a set in $Max(I_R)$. Also, $M$ itself must be in $H$. In particular, if $R_1,...,R_k$ are members of $Max(H)$, where $R_i$ has side length $r_{ij}$ in dimension $j$, then
\begin{equation*}
\left|H\right| \geq 1+\sum_{i=0}^k g_d(r_{i1},...,r_{id})
\end{equation*}
\end{obs}

\section{Proof of the lower bound in Theorem \ref{mincuboids}}
\label{maincuboidproof}
Now we are ready to prove the main theorem. Given $H\in Max(I_M)$, our task is to show that $\left|H\right|\geq \sum_{i=1}^d m_i -(d-1)$. We shall proceed by induction on $d$, and within this by induction on $\sum_{i=1}^d m_i$. First, we establish a slightly stronger result for $d=1$.

\begin{lemma}
\label{onedim}
Let $M=[0,m]\subset \mathbb{R}$ be a line segment. Then every maximal system of cuboid islands in $M$ has size $m$. In particular, $g_1(m)=m$.
\end{lemma}
\begin{proof}
We prove this by induction on $m$. For $m=1$, the result is trivial. For $m\geq 2$ suppose that $H\in Max(I_M)$. There are two different forms that $Max(H)$ can take - either it is a single interval $[0,m-1]$, or $Max(H)=\{[0,a],[a+1,m]\}$ for some integer $1\leq a \leq m-1$. In either case, we use Observation \ref{subsystems} and apply the induction hypothesis to the members of $Max(H)$. This shows that $H$ consists of $m-1$ elements contained in one of the members of $Max(H)$, together with $M$ itself, and so $\left| H \right| =m$.
\end{proof}

In $d$ dimensions, our base case is when any side length $m_i$ of $M$ is $1$. In this case, the problem reduces immediately to the $(d-1)$-dimensional case. Using this, we shall assume that $m_i\geq 2$ for all $i$, and that the theorem holds whenever $\sum_{i=1}^dm_i$ is reduced. We shall now proceed in three different ways, depending on the configuration of $Max(H)$ inside $M$. The first two cases deal with special configurations which can arise when $\left|Max(H)\right|$ is small.

\subsection{Case 1: $ \left|Max(H)=1\right|$}
Without loss of generality, 
\begin{equation*}
Max(H)=\{[0,m_1-1]\times [0,m_2]...\times [0,m_d]\}
\end{equation*}
Applying the induction hypothesis to the sole member of $Max(H)$, and using Observation \ref{subsystems}, we find that
\begin{equation*}
\left|H\right| \geq 1+g_d(m_1-1,m_2,...,m_d) = \sum_{i=1}^d m_i +(d-1)
\end{equation*}
We note that in this case we can get equality.

\subsection{Case 2: $\left|Max(H)\right|>1$, and $Max(H)$ has an element which divides $M$ into $2$ or more regions}
Let $R$ be a member of $Max(H)$ which divides $M$, with
\begin{equation*}
R=[r_1,r_2]\times [0,m_2]\times...\times [0,m_d].
\end{equation*}
Then by Corollary \ref{nhds} $r_{1}\neq 1$ and $r_{2}\neq m_{1}-1$. If $r_{1} = 0$, then we use observation \ref{subsystems} and apply the induction hypothesis in $R$ and in $R=[r_{2}+1,m_1]\times [m_2] \times ... [m_d]$, which must be the sole other member of $Max(H)$.This gives
\begin{align*}
\left| H \right|\geq& 1+\left (\sum_{i=1}^d m_i - m_1 +r_{2} -(d-1)\right) \\
&+\left(\sum_{i=1}^d m_i -m_1+ (m_1-r_{2}-1)-(d-1)\right)\\
=& \left(\sum_{i=1}^d m_i -(d-1)\right)+\left(\sum_{i=1}^d m_i - m_1-(d-1)\right).
\end{align*}
As every $m_i$ is at least $2$, this shows that $\left| H\right|$ larger than we claim for Theorem \ref{mincuboids} - and so equality cannot hold in this case.

If, on the other hand, $1<r_1<r_2<n-1$, then we must have that
\begin{equation*}
Max(H) = \{R, [r_2+1,m_1]\times [m_2] \times ... [m_d],[0,r_1-1]\times [m_2] \times ... [m_d]\}
\end{equation*} 
Using Observation \ref{subsystems} and applying the induction hypothesis to each of these three bricks, we find that
\begin{align*}
\left| H \right| \geq \left(\sum_{i=1}^d m_i -(d-1)\right)+\left(2\sum_{i=1}^d m_i -2 m_1-2(d-1)-1\right).
\end{align*}
Again, using the fact that each $m_i$ is at least $2$, this gives the bound we require for $\left| H\right|$ with strict inequality.

\subsection{Case 3: $\left|Max(H)\right|>1$, and no element of $Max(H)$ divides $M$ into $2$ regions}
We define a path $P$ around some edges of the cuboid by
\begin{align*}
P = \{(0,0,...,0,x_i,m_{i+1},...,m_d):1\leq i \leq d, 0\leq x_i\leq m_i\}\\
 \cup \{(m_{1},m_{2},...,m_{i-1},x_i,0,...,0):1\leq i \leq d, 0\leq x_i\leq m_i\}
\end{align*}
We note that $P$ has two edges in each direction, and that these edges are diametrically opposite each other in $M$. Hence no brick in $Max(H)$ intersects both of these edges, or else it would divide $M$. The length of $P$ is $2\sum_{i=1}^d m_{i}$, and $P$ has $2d$ corners with $2$ edges incident at each.

Now, consider all the members of $Max(H)$ which intersect $P$. Suppose there are $k$ of them, $A_{1},...,A_{k}$, with the $j$th dimension edge length of $A_{i}$ being denoted $a_{ij}$. By Corollary \ref{gaps}, the gaps between consecutive bricks on $P$ are at most $2$. Writing $n_2$ as the number of gaps of length $2$, Corollary \ref{gaps} tells us that at least $n_{2}$ of the $A_{i}$ are not elementary cubes (eg. the ones after the gaps of length $2$). Now, the edges of the bricks which lie on $P$ have total length 
\begin{equation*}
2\sum_{i=1}^d m_{i} - k - n_2
\end{equation*}
Also, there are $k + 2d$ such edges (as there are $2d$ corners in $P$). Hence the $A_{i}$ have between them $k (d-1) - 2d$ edges which are not on $P$ - and so we have that
\begin{equation*}
\sum_{i=1}^k\sum_{j=1}^d a_{ij} \geq 2\sum_{i=1}^d m_{i} - k - n_2 + k(d-1)-2d
\end{equation*}
Now, using Observation \ref{subsystems} and applying the inductive hypothesis in each $A_{i}$, we obtain
\begin{align*}
\left|H\right| &\geq 1+ \sum_{i=1}^k g_d(a_{i1},...,a_{id}) \geq \sum_{i=1}^k\sum_{j=1}^d a_{ij} - k(d-1) +1\\
&\geq 2\sum_{i=1}^d m_{i} - k - n_2 -2d + 1 \\
&= \left(\sum_{i=1}^d m_{i} - d+1\right)+\left(\sum_{i=1}^d m_{i} - k - n_2 -d\right).
\end{align*}
Since the first bracket is the bound we wish to establish for $H$, this is establishes the theorem unless 
\begin{equation*}
k+n_2> \sum_{i=1}^d m_{i} -d.
\end{equation*}
In this case, we observe that $H$ contains each of the $k$ bricks $A_i$, at least one further brick contained in each $A_i$ which is not an elementary cube, and $M$ itself. Since there are at least $n_2$ bricks $A_i$ which are not elementary cubes, we get that
\begin{equation*}
\left|H\right| \geq k+n_2+1\geq\sum_{i=1}^dm_i-d+2.
\end{equation*}
This shows that in this final case Theorem \ref{mincuboids} holds with strict inequality. \qed

\section{Classification of extremal examples for Theorem \ref{mincuboids}}
\label{extremal}
When we showed the upper bound for $g_d(m_1,...,m_d)$, we gave one example of a smallest possible maximal system. In this section we classify all such systems.
\begin{lemma}
\label{equality}
Let $M=[0,m_1]\times...\times [0,m_d]$ and let $H\in Max(I_M)$. If $\left|H\right|$ is minimal among members of $Max(I_M)$, $d\geq 2$ and $m_i \geq 2$ for at least $2$ choices of $i$, then $\left|Max(H)\right|=1$.
\end{lemma}
\begin{proof}
We first note that if $m_d=1$, maximal systems of brick islands in $M$ are precisely those in $M'=[0,m_1]\times...\times [0,m_{d-1}]$. Using this, we can work instead in the cuboid given by projecting in all dimensions where the side length of $M$ is $1$. So we shall assume that $m_i\geq 2$ for each $1\leq i\leq d$.
When $d=2$, this was proved by Lengv\'{a}rszky \cite{Len1}. Examining the proof of Theorem \ref{mincuboids} we note that for equality to hold we must have the following constraints on $H$:
\begin{itemize}
\item $Max(H)=\{A_i:1\leq i \leq k\} $
\item Every $A_i$ is an elementary cube or a brick with all sides of length $1$ except for one side of length $2$.
\item If some side length $a_{ij}$ of $A_i$ is greater than $1$, then some side of $A_i$ that lies along $P$ must be in direction $j$.
\end{itemize}
From these last two constraints we can deduce that every elementary cube contained in some $A_i$ lies on an edge of $P$. If $d\geq 3$ and $m_i\geq 3$ for all $1\leq i\leq d$, let $v$ be some vertex of $M$ which is not on $P$, and $C_v$ be the elementary cube which contains it. Then no brick $A_i$ intersects $C_v$. However, by the first constraint there are no other members of $Max(H)$; hence we can add $C_v$ to $H$, contradicting the maximality of $H$. This contradiction establishes the lemma whenever $d\geq 3$ and $m_i\geq 3$ for all $1\leq i\leq d$.

So we may assume that $d\geq 3$, and that $m_d=2$. In this case we define sets of bricks $H_1$, $H_2$, $H_{12}$ in $d-1$ dimensions by writing
\begin{equation*}
H = \{R \times [0,1]:R\in H_1\}\cup\{R \times[1,2]:R\in H_2\}\cup \{R \times [0,2]:R\in H_{12}\}
\end{equation*}
We note that no element of $H_1$ intersects an element of $H_2$.  Now, $H_1\cup H_2\cup H_{12}$ is a maximal system of brick islands in the $(d-1)$-dimensional cuboid $M'=[0,m_1]\times...\times [0,m_{d-1}]$, and so by Theorem \ref{mincuboids}
\begin{equation*}
\left|H_1 \cup H_2 \cup H_{12}\right|\geq\sum_{i=1}^{d-1} m_i -d+2= \sum_{i=1}^{d} m_i -d.\end{equation*}
Thus if we have equality in Theorem \ref{mincuboids} for $H$, then there is at most one intersection between any of $H_1$, $H_2$ and $H_{12}$. We observe that any minimal member of $H_{12}$ must be in $H_1\cup H_2$, and any maximal member of $H_1\cup H_2$ must be in $H_{12}$. So for equality to hold, $H_{12}$ has a unique minimal element $R$, which is also the unique maximal element of $H_1\cup H_2$. We also know that $H_{12}$ has a unique maximal element $M'$ corresponding to $M\in H$, and so the bricks in $H_{12}$ must be nested. If $\left|H_{12}\right|\geq 2$, then the second largest element of $H_{12}$ corresponds in $H$ to the unique element of $Max(H)$; if $\left|H_{12}\right|=1$, then $R$ in $H_1\cup H_2$ corresponds in $H$ to the unique element of $Max(H)$.
\end{proof}

Using this lemma, we can classify the minimal elements of $Max(I_M)$. A system of brick islands is a minimal element of $Max(I_M)$ if and only if it can be obtained by the following procedure:
\begin{itemize}
\item Take any brick $R$ in $M$ such that every side length of $R$ is $1$ except for one dimension, in which it is $r$.
\item Take any system of $r$ brick islands in $R$ (the largest of which is $R$ itself).
\item All other bricks are nested, with $R$ being the smallest and $M$ being the largest, such that each brick extends the last by $1$ in one direction.
\end{itemize}
We count the bricks in such a system. There are $r$ bricks within $R$, and $\sum_{i=1}^d m_i -(r +d-1)$ to extend each dimension to $m_i$, giving a system of the required size. We prove that these are all the minimal elements of $Max(I_M)$ by induction on $\sum_{i=1}^d m_i$. The base case is when $M$ has at most one dimension of size at least $2$, in which case we can take $R=M$. If $m_i>1$ holds for at least $2$ of the $m_i$, then $Max(H)$ has a unique element $H_{max}$ by lemma \ref{equality}. Applying the induction hypothesis in $H_{max}$, we obtain the result for $M$.

\section{Proof of Theorem \ref{mincubes}}
\label{mincubeproof}
Before we prove our results about systems of cubic islands, we observe the obvious analogue of Observation \ref{subsystems} for cubic islands.
\begin{obs}
\label{cubesubsystems}

Let $C=[0,m]^d$, $H\in Max(I_C)$ and $R\in H$. Then those bricks in $H$ which are contained in $R$ form a set in $Max(I_R)$. Also, $C$ itself must be in $H$. In particular, if $R_1,...,R_k$ are the members of $Max(H)$, where $R_i$ has side length $r_i$
\begin{equation*}
1+\sum_{i=0}^k g'_d(r_i)\leq \left|H\right| \leq 1+\sum_{i=0}^k f'_d(r_i)
\end{equation*}
\end{obs}

Now we prove Theorem \ref{mincubes}, on the minimal size of maximal systems of cubic islands. We wish to show that $g'_d(m)=m$. We first note that $g_d'(m)\geq m$, as a sequence of $m$ nested cubes is maximal in $C=[0,m]^d$. To prove the upper bound for $g'_d(m)$ we will proceed by induction on $m$. For $m\leq 2$ , the theorem is trival. For $d=2$, the theorem was proved in \cite{Len2}. So we shall assume $d\geq 3$ and $m\geq 3$, and that the theorem holds $\forall m' \leq m$. Given $C=[0,m]^d$, and $H\in Max(I_C)$, our task is to show that $\left| H\right| \geq m$. We proceed in three different ways, depending on the size of the largest element of $H$.
\subsection{Case 1: The system $H$ contains an element of size $m-1$ }
In this case, the result follows immediately from the inductive hypothesis, together with Observation \ref{cubesubsystems}.

\subsection{Case 2: The largest element of $H$ is of size $m-2$}
Denote the element of $H$ of size $m-2$ by $R$. Consider the bottom left corner of $R$, with coordinates $(r_1,...,r_d)$, with each of the $r_i$ in $\{0,1,2\}$. Note that they are not all $1$ - if they were we could extend the system $H$ by adding in a cube of size $[0,m-1]^d$. Now, for $1\leq i \leq d$, set
\begin{equation*}
a_i = \left\{
\begin{array}{l l}
  0\text{ if }r_i = 2\\
  m_i-1\text{ if } r_i = 0 \text{ or }1
\end{array}
\right.
\end{equation*}
Then we note that the elementary cube based at $(a_1,...,a_m)$ does not intersect $R$. This shows that, by the maximality of $H$, there is at least one cube $R'$ other than $R$ in $Max(H)$. Thus $H$ contains $C$, $R'$ and $(m-2)$ cubes contained in $R$ (by Observation \ref{cubesubsystems}, and the inductive hypothesis applied to $R$). Consequently $\left|H \right| \geq m$, as required.

\subsection{Case 3: All elements of $H$ are of size at most $m-3$}
Consider the path $P$ as in the proof of Theorem $1$;
\begin{align*}
P = \{(0,0,...,0,x_i,m,...,m):1\leq i \leq d, 0\leq x_i\leq m\}\\
 \cup \{(m,m,...,m,x_i,0,...,0):1\leq i \leq d, 0\leq x_i\leq m\}
\end{align*}
 Given two points $p_1$ and $p_2$ on $P$ which are seperated on $P$ by at least two vertices of $C$, and elementary cubes $C_1$ and $C_2$ containing $p_1$ and $p_2$ respectively, we note that $p_1$ and $p_2$ differ by $m$ in (at least) $1$ dimension. Hence no cube of side at most $m-3$ can intersect both $C_1$ and $C_2$.

Let $A_1,...A_k$ be those cubes in $Max(H)$ which contain a point of the form $p+(c_1,...,c_d)$, with $p\in P$ and $\left| c_i\right|\leq 1$ for $1\leq i\leq d$. We project $A_i$ on to those points $p\in P$ for which $A_i$ has a point of this form. Then each $A_i$ is projected onto at most $2$ edges of $P$ (which occurs precisely when $A_i$ is at most $1$ away from a corner of $P$ in every direction). The gaps between adjacent projections are at most $2$, very similarly to in the cuboid case - if there is a gap of $3$, we can put an elementary cube on $P$ in the middle of it to extend $H$. We may also get extra gaps at the $2d$ corners of $P$, as the cubes closest to the corners need not project into them. These gaps have size at most $2$. Writing $a_i$ for the side length of the cube $A_i$, this gives
\begin{equation*}
2\sum_{i=1}^k a_i + 2k +4d\geq 2md
\end{equation*}
Thus at least one of the following must hold:
\begin{align}
k&\geq m-1 \label{manycubes}\\
\sum_{i=1}^k a_i&\geq m-1 \label{longsides}\\
2d - 4 &\geq m(d-2) \notag
\end{align}
 However, the last inequality does not hold for any pairs of integers $m\geq 3$ and $d\geq 3$. If \ref{manycubes} holds, then we note that $\left| H \right| \geq k+1$, as each of the $A_i$ and $C$ itself are in $H$. If \ref{longsides} holds, we use Observation \ref{cubesubsystems} and apply the inductive hypothesis to each $A_i$ to obtain $\left|H\right| \geq \sum_{i=1}^k a_i +1$. In either case, we get that $\left| H \right| \geq m$. \qed

\section{Proof of Theorem \ref{maxcubes}}
\label{maxcubeproof}
Finally, we prove Theorem \ref{maxcubes}, on the maximal size of a system of cubic islands. As in the setup of the theorem, let $C = [m]^d$ and $H\in Max(I_C)$. Then our task is to show that $\left| H \right| \leq \frac{(m+1)^d-1}{2^d-1}$, and to demonstrate an $H$ for which equality holds when $m=2^k-1$. We do the latter first. For $C_k=[2^k-1]^d$ we define systems of cubic islands $H_k$ recursively. $C_1$ has $H=\{C_1\}$. To form $H_k$, divide $C_k$ into $2^d$ subcubes of side $2^{k-1}-1$ using $d$ hyperplanes passing through the middle of the cube. Place a copy of $H_{k+1}$ in each of these subcubes, and add $C_k$ to obtain $H_k$.  This gives
\begin{align*}
\left|H_k\right| &= 2^d \left|H_{k-1}\right| +1\\
&= 2^d\frac{2^{(k-1)d}-1}{2^d-1}+1\\
&= \frac{2^kd-1}{2^d-1}
\end{align*}
as required.

To show $\left| H \right| \leq \frac{(m+1)^d-1}{2^d-1}$, we use induction on $m$. The result is trivial for $m=1$. Given $m\geq 2$, we split into $2$ cases, depending on the size $Max(H)$.
\subsection{Case 1: $\left|Max(H)\right|=1$ }
Here $Max(H)$ has a unique member $R$ of side length $m-1$. Applying the induction hypothesis in $R$,
\begin{equation*}
\left|H\right| \leq 1+ \frac{m^d-1}{2^d-1}\leq \frac{(m+1)^d-1}{2^d-1}
\end{equation*}
implying the assertion of the theorem.
\subsection{Case 2: $\left|Max(H)\right| \geq 2$}
In this case, $Max(H)$ has no elements of side length greater than $m-2$. Now, order the vertices of $C$ and consider each in turn. If the elementary cube $C_v$ which contains the vertex $v$ intersects no element of $Max(H)$, then add $C_v$ into $H$. If $C_v$ intersects some element $R$ of $Max(H)$ but is not contained in it, then move $R$ into the corner, together with every cube in $H$ that it contains. Note $R$ cannot have contained any other vertex of $C$, as it has side length at most $m-2$. After applying this process to every vertex, we have a family $H'$ with $\left|H'\right|\geq \left|H\right|$, such that every vertex of $C$ is occupied by a different element of $Max(H')$. Hence $\left|Max(H')\right|\geq 2^d$. Now we use the same argument as applied in \cite{Cze} and \cite{Len2}. 

Suppose $A_1,...,A_k$ are the elements of $Max(H')$. Then extend each $A_i$ by $\frac{1}{2}$ in every direction. The interiors of the extended cubes are disjoint, and are contained within an extended version of $C$ of side $m+1$. Thus their total volume $\sum_{i=1}^k (a_i+1)^d$ is at most $(m+1)^d$. We write $n(A_i)$ for the number of cubes of $H'$ which are contained in $A_i$, and $a_i$ for the side length of $A_i$. Using Observation \ref{cubesubsystems}, and applying the inductive hypothesis in each $A_i$, we get that
\begin{align*}
\left|H\right|\leq\left|H'\right| &\leq 1+\sum_{i=1}^k n(A_i) \leq 1+\sum_{i=1}^k\frac{(a_i+1)^d-1}{2^d-1}\\
 & \leq 1+\frac{(m+1)^d}{2^d-1} -\frac{2^d}{2^d-1} =\frac{(n+1)^d-1}{2^d-1}
\end{align*}
This is exactly the bound we want on $\left| H\right|$, and our proof is complete. \qed
\section{Further work}
As mentioned in \cite{Len2}, we could consider the problem of cubic islands in a cuboid; the members of our system $H$ would be cubes, while $M$ remains a cuboid with arbitrary sides. While we have got a best polynomial upper bound on $f_d'(m)$, we have not found a reasonable lower bound, and this is another possible extension.

\end{document}